\documentclass[letterpaper, 10 pt, conference]{ieeeconf}  

\IEEEoverridecommandlockouts                              

\usepackage{amsmath,amssymb,bm,bbm,mathrsfs,amscd, mathtools}
\usepackage{color}
\usepackage{dsfont}
\usepackage{graphicx}
\usepackage{epsfig}
\usepackage{subcaption}
\usepackage{balance}
\usepackage{algorithm}
\usepackage[noend]{algpseudocode}
\usepackage{tikz}
\usepackage{cite}
\usepackage{booktabs,tabularx}
\usepackage{pifont}

\graphicspath{{figures/}}

\newtheorem{theorem}{Theorem}

\newtheorem{lemma}{Lemma}

\newtheorem{remark}{Remark}
\newtheorem{assumption}{Assumption}

\newcommand{\RR}{{\mathbb{R}}}
\newcommand{\NN}{{\mathbb{N}}}
\newcommand{\EE}{{\mathbb{E}}}
\newcommand{\PP}{{\mathbb{P}}}
\newcommand{\JJ}{{\mathbb{J}}}
\newcommand{\FF}{{\mathbb{F}}}

\newcommand{\mc}[1]{\mathcal{#1}}
\newcommand{\norm}[1]{\left\|#1\right\|}
\newcommand{\normsq}[1]{\|#1\|^2}
\newcommand{\EEk}[1]{\EE[#1|\mc F_k]}

\newcommand{\op}{\operatorname}

\newcommand{\bs}{\boldsymbol}
\newcommand{\fineass}{\hfill$\square$}
\newcommand{\cmark}{\ding{51}}%
\newcommand{\xmark}{\ding{55}}%

\title{\LARGE \bf A stochastic generalized Nash equilibrium model\\ for platforms competition in the ride-hail market}

\author{Filippo Fabiani and Barbara Franci
	\thanks{F.~Fabiani is with the Department of Engineering Science, University of Oxford, OX1 3PJ, United Kingdom {\tt \footnotesize (filippo.fabiani@eng.ox.ac.uk)}. B.~Franci is with the Department of Data Science and Knowledge Engineering, Maastricht University, P.O. Box 616, NL\textendash 6200 MD Maastricht, The Netherlands {\tt \footnotesize (b.franci@maastrichtuniversity.nl)}. This work was partially supported through the Government’s modern industrial strategy by Innovate UK, part of UK Research and Innovation, under Project LEO (Ref.~104781).}
}

\begin{document}

\maketitle
\thispagestyle{empty}
\pagestyle{empty}

\begin{abstract}
	The presence of uncertainties in the ride-hailing market complicates the pricing strategies of on-demand platforms that compete each other to offer a mobility service while striving to maximize their profit. Looking at this problem as a stochastic generalized Nash equilibrium problem (SGNEP), we design a distributed, stochastic equilibrium seeking algorithm with Tikhonov regularization to find an optimal pricing strategy. Remarkably, the proposed iterative scheme does not require an increasing (possibly infinite) number of samples of the random variable to perform the stochastic approximation, thus making it appealing from a practical perspective. Moreover, we show that the algorithm returns a Nash equilibrium under mere monotonicity assumption and a careful choice of the step size sequence, obtained by exploiting the specific structure of the SGNEP at hand. We finally corroborate our results on a numerical instance of the on-demand ride-hailing market.
\end{abstract}

\section{Introduction}
In the last few years, we have been experiencing a dizzying growth of the ride-hailing market \cite{nyc_ridehailing}, where on-demand ride-hailing platforms, such as Uber, Lyft and Didi Chuxing, have to settle up suitable pricing strategies to be attractive on two nearly complementary fronts: costumers and drivers. Each ride-hailing firm, indeed, not only competes for costumers with the other firms and with traditional transportation systems, but also strives to secure an as wide as possible fleet of ``loyal'' drivers so that it can meet possibly growing costumers' demand, which often may not be predicted accurately. In this framework, competition among the platforms can be naturally described through a stochastic generalized Nash equilibrium problem (SGNEP): the firms aim at maximizing their expected valued profit function, trying to satisfy the demand for rides, which is typically uncertain, while sharing the market with the other platforms. 

Specifically, SGNEPs amount to a collection of mutually coupled stochastic optimization problems, which are challenging to address especially if one aims at finding a solution, i.e., a stochastic generalized Nash equilibrium (SGNE), in a distributed fashion. The difficulties are mainly due to the constraints coupling the agents' strategies, and the presence of uncertainty. The first issue is typically accommodate by reformulating the problems as a monotone inclusion \cite{yi2019, franci2020fb} obtained by exploiting the Karush-Khun-Tucker conditions of the coupled optimization problems. Concerning the uncertainty instead, the usual approach approximates the expected-valued pseudogradient mapping of the game by leveraging available realizations of the uncertainty \cite{robbins1951,koshal2013}.

As a next step, one should design a sequence of instructions alternating distributed computation and communication steps, i.e., an algorithm, with provable convergence guarantees to an equilibrium solution of the SGNEP at hand. Among the numerous algorithms for classic stochastic optimization \cite{iusem2017,bot2021}, only few of them are amenable to solve SGNEPs. For this class of problems, indeed, we include the forward-backward (SFB) algorithm \cite{franci2020fb} and related variations, such as the relaxed forward-backward (SRFB) \cite{franci2021} or the projected-reflected-gradient (SPRG) algorithms \cite{franci2021sprg}. However, these procedures are affected by common drawbacks: the monotonicity assumption of the operators involved and the number of samples necessary for the approximation. In fact, both SFB and SPRG algorithms converge in case the pseudogradient mapping is strongly monotone or cocoercive, which correspond to rather strong assumptions. Moreover, FB-based methods take as an approximation the average over an increasing, possibly infinite, number of samples of the uncertainty, which is impractical or even unrealistic. These reasons motive us to modify the traditional FB algorithm with a Tikhonov (Tik) regularization method \cite[Ch.~12]{facchinei2007}, tailored for SGNEPs. Introducing a regularization sequence is indeed a well-known technique to weaken strong assumptions \cite{koshal2013, kannan2012}, since it allows one to obtain, e.g., a strongly monotone operator starting from a merely monotone one \cite{facchinei2007}. We summarize all these considerations in Table~\ref{table_algo}.

\begin{table}[b]
\centering
\begin{tabular}{p{2cm}p{.6cm}ccc}
\toprule
        & Tik  & SpFB \cite{franci2020fb}& SPRG \cite{franci2021sprg} &  SRFB  \cite{franci2021}\\ 
        \midrule
\textsc{Monotonicity}    &   \cmark    &  \xmark  &  \xmark    &  \cmark  \\ 
\# \textsc{Sample(s)}   & 1   & $N_k$  & $N_k$ & $N_k$    \\ 
\textsc{Step Size }  & $\alpha_k$   & $\alpha$  & $\alpha$  & $\alpha$   \\ 
\bottomrule
\end{tabular}
\caption{Existing FB-based algorithms for SGNEPs which converge with monotonicity (\cmark) or stronger assumptions (\xmark). The quantity $N_k$ indicates an increasing (possibly infinite) number of samples, while 
$\alpha_k$ indicates a the time-varying step size sequence, as opposed to $\alpha$. 
}\label{table_algo}
\end{table}

In \cite{koshal2013}, a distributed version of this Tikhonov regularization-based algorithm was introduced for SNEPs, i.e., without coupling constraints. We show here that the generalization to SGNEPs is possible, albeit non-trivial, as the time-varying nature of both the step sizes and regularization step poses technical challenges to be treated carefully when using operator splitting techniques (\S \ref{sec_discussion}). We can hence summarize our contributions as follows:
\begin{itemize}
\item Inspired by the cognate literature, we propose a noncooperative model for on-demand competing ride-hailing platforms under a regulated pricing scenario and uncertainties, and we recast it as a SGNEP (\S \ref{sec:app}, \ref{sec_GNEPs}); 
\item We propose a distributed Tikhonov regularization-based algorithm that leverages a finite number of samples of the uncertainty to perform the stochastic approximation, thus circumventing a crucial issue in equilibrium seeking algorithm design for SGNEPs (\S \ref{sec_algo_sgnep});
\item We show that the algorithm converges to a SGNE i) under mere monotonicity of the pseudogradient mapping, one of the weakest assumptions to establish convergence \cite{facchinei2010}, and ii) with a careful choice of the step size sequence that exploits the structure of the problem.
\end{itemize}
In conclusion, the performance of the designed Tikhonov-like algorithm is tested on a numerical instance of the proposed ride-hailing competition market model (\S \ref{sec:num_sim}).

\subsection{Notation and Preliminaries}
\paragraph{Notation}
$\NN$ indicates the set of natural numbers and
$\RR$ ($\bar\RR=\RR\cup\{\infty\}$) is the set of (extended) real numbers.
$\langle\cdot,\cdot\rangle:\RR^n\times\RR^n\to\RR$ denotes the standard inner product and $\norm{\cdot}$ is the associated Euclidean norm.
Given a vector $x\in\RR^n$, $x_{\textrm{min}}=\textrm{min}_{i=1,\dots,n}x_i$. 
We indicate that a matrix $A$ is positive definite, i.e., $x^\top Ax>0$, with $A\succ0$. 
Given a symmetric $W\succ0$, the $W$-induced inner product is $\langle x, y\rangle_{W}=\langle W x, y\rangle$ and the associated norm is defined as $\norm{x}_{W}=\sqrt{\langle W x, x\rangle}$. $\op{Id}$ is the identity operator. 
$\iota_{\mc X}$ is the indicator function of the set $\mc X$, that is, $\iota_{\mc X}(x)=0$ if $x\in \mc X$ and $\iota_{\mc X}(x)=\infty$ otherwise. The set-valued mapping $\mathrm{N}_{\mc X} : \RR^{n} \to \RR^{n}$ denotes the normal cone operator of the set $\mc X$, i.e., $\mathrm{N}_{\mc X}(x)=\emptyset$ if $x \notin \mc X$, $\mathrm{N}_{\mc X}(x)=\left\{v \in \RR^{n} | \sup _{z \in \mc X} \langle v,z-x\rangle \leq 0\right\}$ otherwise.

\paragraph{Operator theory}
Let $\op{gra}(F)=\{(x,u):u\in F(x)\}$ be the graph of $F : \mc X \subseteq \RR^{n} \to \RR^{n}$. Then, F is said to be: monotone on $\mc X$ if
$\langle F(x)-F(y),x-y\rangle \geq 0, \text{ for all } x, y \in \mc X;$ $\mu$-strongly monotone on $\mc X$ if there exists a constant $\mu>0$ such that
$\langle F(x)-F(y),x-y\rangle \geq \mu\|x-y\|^{2}, \text{ for all } x, y \in \mc X;$ maximally monotone if there exists no monotone operator 
$G : \mc X \to \RR^n$ such that $\operatorname{gra} G$ properly contains $\operatorname{gra} F$: 
$\ell$-Lipschitz continuous with constant $\ell>0$ if for all $x, y\in \mc X$ 
$\|F(x)-F(y)\| \leq \ell\|x-y\|$.
The projection operator onto $\mc C$ is the operator defined as $\op{proj}_{\mc C}(x)={\op{argmin}_{z \in\mc  C}}\normsq{z-x}.$ 

\paragraph{Graph theory}

The weighted adjacency matrix associated to $\mc G^\lambda$ is denoted with $W=[w_{ij}]_{i,j\in\mc I}\in\RR^{N\times N}$, where $w_{ij}>0$ if agents $i$ and $j$ can communicate with each other and $w_{ij}=0$ otherwise. Then, letting $D=\op{diag}\{d_1,\dots,d_N\}$ where $d_{i,h}=\sum_{j=1}^Nw_{ij}$ is the degree of agent $i$, the associated Laplacian is given by $L=D-W\in\RR^{N\times N}$. It follows from Assumption \ref{ass_graph} that $L=L^\top$.

\paragraph{Stochastic setting}
$\EE_\xi$ represent the mathematical expectation with respect to the distribution of the random variable $\xi(\omega)$, in the probability space $(\Xi, \mc F, \PP)$. We avoid expressing the dependency on $\omega$ when clear from the context and often simply write $\EE$. We use i.i.d. to indicate independent identically distributed random variables.

\section{On-demand competing ride-hailing firms}\label{sec:app}
We take inspiration from \cite{zhong2022demand,he2020off,bimpikis2019spatial} to examine how $N \in \NN$ on-demand competing ride-hailing platforms (e.g., Uber, Didi Chuxing, Lyft, Juno and Via) design their pricing strategies under a regulated pricing scenario and the presence of uncertainties. Specifically, given a continuum of potential riders of mass $C_h > 0$ for each area of interest $h \in \mc{H} \subset \NN$ (e.g., suburbs, city centres, airports), each firm $i \in \mc{I} \coloneqq \{1, \ldots, N\}$ aims at maximizing its profit by setting i) a price $p_{i,h} \geq 0$ for the on-demand ride-hailing service to attract as many costumers as possible in the $h$-th area, with
\begin{equation}\label{eq:equity_price}
	{\frac{1}{N} \sum_{i \in \mc{I}} p_{i,h} \leq \bar{p}_h, \ \forall h \in \mc{H},}
\end{equation}
capping the averaged maximum price allowed $\bar{p}_h > 0$, usually imposed by consumers' associations; ii) a wage $w_{i,h} \geq \underline{w} > 0$ for the registered drivers on the $i$-th ride-hailing platform, which is also typically regulated by institutions \cite{beer2017qualitative,zhong2022demand}, to meet the resulting costumers' demand. As one may expect, since profit maximization is a consideration, $p_{i,h}$ and $w_{i,h}$ shall be necessarily interdependent. 

Similar to \cite{mcguire1983industry,he2020off}, we assume that the fraction of customers who choose the $i$-th platform’s service in the $h$-th area, i.e., the demand for the $i$-th firm, is characterized as:
\begin{equation}\label{eq:cust_demand}
		{d_{i,h} = \frac{C_h K_{i,h}}{ \bar{p} \sum_{j \in \mc{I}} K_{j,h}}  \left(\bar{p} - p_{i,h} + \frac{\theta_i}{N-1} \sum_{j \in \mc{I}\setminus\{i\}} p_{j,h} \right),}
\end{equation}
where $K_{i,h} > 0$ denotes the number of registered drivers on the $i$-th ride-hailing platform who prefer to work in the $h$-th area, $\bar{p} > \textrm{max}_{h \in \mc{H}} \ \bar{p}_h$ is a maximum service price, and $\theta_i \in [0,1]$ models the substitutability of the service provided by each firm. When $\theta_i$ is close to $0$, the service of the $i$-th platform is almost independent from the others, while $\theta_i$ close to $1$ means that it is fully substitutable, thus obtaining a perfect competition market. We discuss in details the role played by this parameter in \S \ref{sec:num_sim} with a numerical example.

Note that, however, the demand request in \eqref{eq:cust_demand} does not account for the willingness of the drivers to actually provide a service, which is key to meet the costumers' demand and hence maximize the profit. In fact, any driver provides service in a prescribed area $h \in \mc{H}$ only if its earn is greater than their \emph{opportunity cost}, here denoted by $\delta_{i,h}$, which we let coincide with a random variable. For each firm and area, we assume the wage be given by $w_{i,h} = \beta p_{i,h}$, where the parameter $\beta > 0$ denotes the commission ratio that the platform should pay to its driver, typically regulated by a third party (e.g., governments \cite{zhong2022demand}). Thus, for all $i \in \mc{I}$ and $h \in \mc{H}$, we introduce the \emph{effective demand} as
\begin{equation}\label{eq:cust_demand_eff}
	d^\textrm{e}_{i,h} = d_{i,h} \ \PP[w_{i,h} \geq \delta_{i,h}],
\end{equation}
which coincides with the portion of costumers' demand for which a ride-haling company can actually claim a payment. Roughly speaking, a driver is willing to provide a service only if the wage she gets matches (at least) her expectations. This directly leads us to define the fraction of drivers that give services on the $i$-th platform as:
\begin{equation}\label{eq:driv_service}
	k_{i,h} = K_{i,h} \ \PP[w_{i,h} \geq \delta_{i,h}].
\end{equation}

As a consequence, the effective demand in \eqref{eq:cust_demand_eff} can be equivalently obtained from \eqref{eq:cust_demand} by replacing $K_{i,h}$ with $k_{i,h}$.
However, the costumers' service request in \eqref{eq:cust_demand} may be affected by an additional source of uncertainty. In fact, it is unlikely that the $i$-th firm is aware of the total number of potential drivers, $\sum_{j \in \mc{I}} K_{j,h}$, both for privacy reasons and possible multiple registrations. 
Thus, we define $\mc C_h(\xi)=\frac{C_h(\xi)}{\sum_{j \in \mc{I}} K_{j,h}}$ as the unknown fraction of passengers in area $h\in\mc H$, which allows us to explicitly account for the uncertain parameter $\xi$. 
In accordance, the demands in \eqref{eq:cust_demand} and \eqref{eq:cust_demand_eff} turn into random variables $d_{i,h}(\xi)$ and $d^\textrm{e}_{i,h}(\xi)$.

As commonly  adopted in the literature \cite{zhong2022demand,bimpikis2019spatial}, we restrict attention to the case where the drivers’ willingness is uniformly distributed in $[\underline{w}, \, \bar{w}_{i,h}]$, so that $\PP[w_{i,h} \geq \delta_{i,h}] = (w_{i,h} - \underline{w})/(\bar{w}_{i,h} - \underline{w})$. This not only allows us to consider one source of uncertainty, but also to make the constraint in \eqref{eq:driv_service} convex. Note that a similar argument can be adopted for any distribution concave in the induced demand $d_{i,h}$ (e.g., exponential, Pareto).  
Then, the stochastic optimization problem associated to each platform amounts to:
\begin{equation}\label{eq:plat_game}
	\forall i \in \mc I: \left\{
	\begin{aligned}
		&\underset{(p_{i,h}, w_{i,h})_{h \in \mc{H}}}{\textrm{max}} &&\EE_\xi[\sum\limits_{h \in \mc{H}} (p_{i,h} d^{\textrm{e}}_{i,h}(\xi) - w_{i,h} k_{i,h})]\\ 
		& \hspace{.6cm}\text { s.t. } && \eqref{eq:equity_price}, \eqref{eq:cust_demand_eff}, \eqref{eq:driv_service}, p_{i,h} \geq 0, \ \forall h \in \mc{H},\\
		&&& w_{i,h} \in [\underline{w}, \, \bar{w}_{i,h}], \ \forall h \in \mc{H}.
	\end{aligned}
	\right.
\end{equation}
The first part of the cost function amounts to the profit of the $i$-th firm to provide a service to the costumers, while the second one considers the costs for providing a service to the drivers. Unlike \cite{mcguire1983industry,he2020off,zhong2022demand}, the cost functions in \eqref{eq:plat_game} accounts for the number of actual drivers who decide to provide a service rather than the whole fleet of registered ones, $K_{i,h}$. 
After replacing the equality constraints in \eqref{eq:cust_demand_eff} and \eqref{eq:driv_service}, we obtain a collection of mutually coupled stochastic optimization problems, where the cost functions (cubic in $p_{i,h}$, due to the distribution of $\delta_{i,h}$) are affected by the uncertain fraction of potential riders, i.e., $\mc C_h(\xi)$. 
The proposed model relies on a common assumption in the literature: those riders who do not get assigned to a driver when they seek service from the ride-haling platforms, e.g., because of excess demand for rides, use a different mean of transportation. The need for a common platform handling competition in ride-haling mobility to satisfy the costumers' demand has been indeed recently explored in, e.g., \cite{PANDEY2019269,fabiani2021personalized}.

Throughout the paper we treat this model as a SGNEP, and we propose an algorithm to compute an equilibrium solution, according to the discussion presented in the next section.

\section{Stochastic generalized Nash equilibrium problem}\label{sec_GNEPs}

To compact the notation, we rewrite the problem in \eqref{eq:plat_game} as
\begin{equation}\label{eq_game}
\forall i \in \mc I: \quad\begin{cases}
\textrm{min}_{x_i \in \Omega_i}& \JJ_i\left(x_i, \bs{x}_{-i}\right)\\ 
\text { s.t. } & g(x_i,\bs x_{-i})\leq0.
\end{cases}
\end{equation}
where $x_i=\op{col}((p_{i,h})_{h \in \mc{H}}) \in \RR^{n}$, $n \coloneqq |\mc{H}|$, $\bs x=\op{col}((x_i)_{i\in\mc I}) \in \RR^{nN}$, and $\bs x_{-i}=\op{col}((x_j)_{j\neq i})$. Moreover, we indicate the set of local constraints of firm $i$ as $\Omega_i\in\RR^{n}$ and let $\bs \Omega\coloneqq\prod_{i\in\mc I}\Omega_i$. On the other hand, the set of coupling constraints arising from \eqref{eq:equity_price} in a general form read as
\begin{equation}\label{collective_set}
\bs{\mc{X}}\coloneqq\bs\Omega\cap\{\bs y \in\bs\RR^{nN}\; | \;g(\bs y) \leq {\bf{0}}_{m}\},
\end{equation}
where $g:\RR^{nN}\to\RR^m$. We indicate with $\bs{\mc X}_i(\bs x_{-i})$ the piece of coupling constraints corresponding to agent $i$, which is affected by the decision variables of the other agents $\bs x_{-i}$. We stress that the formulation of the SGNEP in \eqref{eq_game} is standard \cite{yi2019,franci2021,belgioioso2017}, and hence the theory we develop applies to all the SGNEPs satisfying the assumptions introduced next. 

\begin{assumption}[Constraint qualification]\label{ass_constr}
For each $i \in \mc I,$ the set $\Omega_{i}$ is nonempty, closed and convex. The set $\bs{\mc{X}}$ satisfies Slater's constraint qualification. 
\fineass
\end{assumption}
\begin{assumption}[Separable convex coupling constraints] 
The mapping $g$ in \eqref{collective_set} has a separable form, i.e., $g(\boldsymbol{x})\coloneqq\sum_{i=1}^{N} g_{i}(x_{i})$, for some convex differentiable functions $g_{i}:\mathbb{R}^{n} \rightarrow \mathbb{R}^{m}$, $i\in\mc I$ and it is $\ell_{\mathrm{g}}$-Lipschitz continuous. Its gradient $\nabla g$ is bounded, i.e., $\op{sup}_{\bs x\in\bs{\mc X}}\|\nabla g(\boldsymbol{x})\| \leq B_{\nabla \mathrm{g}}$.
\fineass
\end{assumption}

Given the stochastic nature of our collection of problems, we indicate the cost function of each agent $i\in\mc I$ as 
\begin{equation}\label{eq_cost_stoc}
\JJ_i(x_i,\bs{x}_{-i})\coloneqq\EE[J_i(x_i,\bs{x}_{-i},\xi(\omega))],
\end{equation}
for some measurable function $J_i:\mc \RR^{n}\times \RR^d\to \RR$. We assume that $\EE[J_i(\bs{x},\xi)]$ is well defined for all feasible $\bs{x}\in\bs{\mc X}$ \cite{ravat2011}. 

\begin{assumption}[Cost function convexity]\label{ass_J}
For each $i \in \mc I$ and $\boldsymbol{x}_{-i} \in \bs{\mc{X}}_{-i}$ the function $\JJ_{i}(\cdot, \boldsymbol{x}_{-i})$ is convex and continuously differentiable.
\fineass\end{assumption}

The goal of the firms is hence to solve \eqref{eq_game} to find a SGNE, i.e., a strategy profile where no agent can decrease its cost function by unilaterally deviating from its decision. Formally, a SGNE is a collective vector $\bs x^*\in\bs{\mc X}$ such that for all $i \in \mc I$
$$\JJ_i(x_i^{*}, \boldsymbol x_{-i}^{*}) \leq \inf \{\JJ_i(y, \boldsymbol x_{-i}^{*})\; | \; y \in \mc{X}_i(\bs x^\ast_{-i})\}.$$

Existence of a SGNE for the game in \eqref{eq_game} is guaranteed under suitable assumptions \cite[\S 3.1]{ravat2011}, though uniqueness does not hold in general \cite[\S 3.2]{ravat2011}. 
Within all possible Nash equilibria, we focus on those coinciding with the solutions of a suitable stochastic variational inequality (SVI) \cite{koshal2013}. 
Thus, we introduce the pseudogradient mapping of the game as
\begin{equation}\label{eq_grad}
\FF(\bs{x})\coloneqq\op{col}\left((\EE[\nabla_{x_{i}} J_{i}(x_{i}, \bs{x}_{-i})])_{i \in \mc{I}}\right).
\end{equation} 
Flipping the expected value and the gradient follows from differentiability of $J_i(\cdot,\bs{x}_{-i})$ (Assumption \ref{ass_J}) \cite[Th.~7.44]{shapiro2021}.
Then, the SVI associated to the SGNEP in \eqref{eq_game} reads as
\begin{equation}\label{eq_SVI}
\langle \FF(\bs x^*),\bs x-\bs x^*\rangle
\geq 0,\text { for all } \bs x \in \bs{\mc X}.
\end{equation}
If Assumptions \ref{ass_constr}--\ref{ass_J} hold, any solution to $\op{SVI}(\bs{\mc X} , \FF)$ in \eqref{eq_SVI} is a SGNE of the game in (\ref{eq_game}), while the converse does not necessarily hold. A game may have a Nash equilibrium while the associated (S)VI may have no solution \cite[Prop.~12.7]{palomar2010}. 
\begin{assumption}[Existence of a variational equilibrium]\label{ass_sol}
The SVI in \eqref{eq_SVI} has at least one solution.
\fineass
\end{assumption}
We call variational equilibria (v-SGNE) the SGNE that are also solution to $\op{SVI}(\bs{\mc X} , \FF)$ in (\ref{eq_SVI}) with $\FF$ in (\ref{eq_grad}) and $\bs{\mc X}$ in (\ref{collective_set}). 
These equilibria can be characterized in terms of the Karush--Kuhn--Tucker (KKT) conditions of the coupled optimization problems in (\ref{eq_game}), i.e.,
a $\bs x^{*}$ is a v-SGNE if and only if the following inclusion is satisfied for $\lambda\in\RR^m_{\geq0}$ \cite[Th.~4.6]{facchinei2010}:
\begin{equation}\label{eq_T}
0\in\mc T(\bs{x},\bs\lambda)\coloneqq\left[\begin{array}{c}
\FF(\bs{x})+\op{N}_{\bs \Omega}(\bs x)+\nabla g(\bs x)^{\top}\bs \lambda \\ 
\mathrm{N}_{\RR_{ \geq 0}^{m}}(\bs\lambda)-g(\bs x)
\end{array}\right],
\end{equation}
where $\mc T:\bs{\mc{X}}\times \RR^m_{\geq 0}\rightrightarrows \RR^{nN}\times\RR^m$ is a set-valued mapping. 
According to \cite[Th.~3.1]{facchinei2007vi}, \cite[Th~3.1]{auslender2000}, the v-SGNE are those equilibria such that the shared constraints have the same dual variable for all the agents, i.e., $\lambda_i=\lambda$ for all $i\in\mc I$, and solve the $\op{SVI}(\bs{\mc X},\FF)$ in \eqref{eq_SVI}.
Then, the v-SGNE of the game in \eqref{eq_game} correspond to the zeros of $\mc T$, which can be split as the sum of two operators, $\mc T=\mc A+\mc B$, where
\begin{equation}\label{eq_splitting}
\begin{aligned}
\mc{A} &:\left[\begin{array}{l}
\bs{x} \\
\lambda
\end{array}\right] \mapsto\left[\begin{array}{c}
\FF(\bs{x}) \\ 
0
\end{array}\right]+\left[\begin{array}{c}
\nabla g(\bs x)^\top\bs \lambda\\
-g(\bs x)
\end{array}\right],\\
\mc{B} &:\left[\begin{array}{l}
\bs{x} \\ 
\lambda
\end{array}\right] \mapsto\left[\begin{array}{c}
\op N_{\bs \Omega}(\bs x) \\ 
\op N_{\RR_{ \geq 0}^{m}}(\bs \lambda)
\end{array}\right].
\end{aligned}
\end{equation}

\section{Distributed stochastic Tikhonov relaxation}\label{sec_algo_sgnep}

\begin{algorithm}[t]
	\caption{Distributed stochastic Tikhonov relaxation}\label{algo_i}
	\smallskip
	\textbf{Initialization}: $x_i^0 \in \Omega_i, \lambda_i^0 \in \RR_{\geq0}^{m},$ and $z_i^0 \in \RR^{m} .$\\
	\smallskip
	\textbf{Iteration} $k$: Agent $i$ receives $x_{j}^k$ for all $j \in \mathcal{N}_{i}^{J}$ and $z_j^k, \lambda_{j}^k$ for $j \in \mathcal{N}_{i}^{\lambda}$, then updates:
	\smallskip
	\begin{equation*}\label{eq_prox}
		\begin{aligned}
			x_i^{k+1}&=\op{proj}_{\Omega_i}\{ x_i^k-\alpha^k_{i}\gamma_i(\hat F_{i}(x_i^k, \boldsymbol{x}_{-i}^k,\xi_i^k) +\nabla g_i(x_i)^\top \lambda_i^k \\
			&\hspace{6.5cm} +\epsilon^k_ix^k_i)\}\\
			z_i^{k+1}&= z_i^k-\alpha_i^k\nu_{i}( \textstyle{\sum_{j \in \mathcal{N}_{i}^{\lambda}}} w_{ij}(\lambda_i^k-\lambda_{j}^k)+\epsilon^k_iz_i^k)\\
			\lambda_i^{k+1}&=\op{proj}_{\RR^m_{\geq 0}}\{ \lambda_i^k+\alpha_k^i\tau_{i}(g_i(x_i^k) -\epsilon^k_i\lambda_k^i)\\
			&\hspace{1.5cm}+\alpha_k^i\tau_i
			\textstyle{\sum_{j \in \mathcal{N}_{i}^{\lambda}}} w_{ij}[(z_{i}^{k}-z_j^k)-
			(\lambda_i^k-\lambda_j^k)]\}
		\end{aligned}
	\end{equation*}
\end{algorithm}

We now discuss in details the sequence of instructions summarized in Algorithm \ref{algo_i}. For the local decision variable $x_i$ the projection onto $\Omega_i$ guarantees that the local constraints are always satisfied, while the coupling constraints are enforced asymptotically through the (nonnegative, due to the projection onto $\RR^m_{\geq0}$) dual variable $\lambda_i$.
The auxiliary variable $z_i$, instead, forces consensus on the dual variables \cite{yi2019,facchinei2007vi}.

We assume that the decision-maker $i$ knows its feasible set $\Omega_i$, and its part of the coupling constraints $\bs{\mc X}_i(\bs x_{-i})$.
The set of agents $j$ whose decision variables affect the cost function of agent $i$, are denoted by $\mc N_i^J$. Specifically, some $j \in \mc{I}$ belongs to $\mc N_i^J$ if $J_i(x_i,\bs x_{-i})$ explicitly depends on $x_j$.

Let us then introduce the graph $\mc G^\lambda=(\mc I,\mc E^\lambda)$ through which a local copy of the dual variable is shared, along with of the auxiliary one, $z_i\in\RR^m$. The set of edges $\mc E^\lambda$ of the multiplier graph $\mc G^\lambda$, is given by: $(i,j)\in\mc E^\lambda$ if player $j$ share its $\{\lambda_j,z_j\}$ with player $i$. For all $i\in\mc I$, the neighboring agents in $\mc G^\lambda$ form the set $\mc N^\lambda_i=\{j\in\mc I:(i,j)\in\mc E^\lambda\}$. Under these premises, Algorithm \ref{algo_i} is distributed in the sense that each agent knows its own problem data and communicates with the other agents through $\mc G^\lambda$. To guarantee that consensus can be reached, we make the following assumption.
\begin{assumption}[Graph connectivity]\label{ass_graph}
The multiplier graph $\mc G^\lambda$ is undirected and connected.
\fineass\end{assumption}

By making use of $\mc G^\lambda$, we hence note that consensus on the dual variable can be enforced via equality constraint ${\bf{ L}}\bs\lambda=0$, where ${\bf{L}}=L\otimes \op{Id}_m\in\RR^{Nm\times Nm}$, $L$ being the Laplacian of the graph, and $\bs \lambda=\op{col}(\lambda_1,\dots,\lambda_N)\in\RR^{Nm}$. Following \cite{yi2019}, the operators $\mc A$ and $\mc B$ in \eqref{eq_splitting} can thus be extended to
\begin{equation}\label{eq_expanded}
\begin{aligned}
\bar{\mc{A}} &:\left[\begin{array}{l}
\bs{x} \\
\bs z\\
\bs \lambda
\end{array}\right] \mapsto\left[\begin{array}{c}
\FF(\bs{x}) \\ 
0\\
{\bf{L}}\bs\lambda
\end{array}\right]+
\left[\begin{array}{ccc}
\nabla G(\bs x)^\top\bs\lambda\\
{\bf{L}} \bs \lambda\\
-G(\bs x)-{\bf{L}} \bs z
\end{array}\right],\\
\bar{\mc{B}} &:\left[\begin{array}{l}
\bs{x} \\ 
\bs z\\
\bs \lambda
\end{array}\right] \mapsto\left[\begin{array}{c}
\op N_{\bs \Omega}(\bs x) \\ 
{\bf{0}}\\
\mathrm{N}_{\RR_{ \geq 0}^{m}}(\bs \lambda)
\end{array}\right],
\end{aligned}
\end{equation}
where $\bs z=\op{col}(z_1,\dots,z_N)\in\RR^{Nm}$, $G(\bs x)=\op{diag}((g_i(x_i))_{i\in\mc I})$ and $\nabla G(\bs x)=\op{diag}((\nabla_{x_i}g_i(x_i))_{i\in\mc I})$.
From now on, we indicate the local decision variable of each agent taking part to the SGNEP as $\bs u=\op{col}(\bs x,\bs z,\bs \lambda)$. 

\begin{remark}\label{remark_zeros}
It can be proven that if Assumptions \ref{ass_constr}--\ref{ass_graph} are satisfied, then the following hold \cite{yi2019, franci2021,franci2020fb}:
\begin{enumerate}
\item[(i)] Given any $\bs u^* \in \op{zer}(\bar{\mc A}+\bar{\mc B})$, $\bs{x}^{*}$ is a v-SGNE of game in (\ref{eq_game}), i.e., $\bs{x}^*$ solves the $\op{SVI}(\bs{\mc X} , \FF)$ in (\ref{eq_SVI}), $\bs{\lambda}^{*}=\mathbf{1}_{N} \otimes \lambda^{*},$ and $(\bs{x}^*,\lambda^{*})$ satisfy the KKT condition in (\ref{eq_T}), i.e., $\op{col}(\bs{x}^*, \lambda^{*}) \in \op{zer}(\mc A+\mc B)$.
\item[(ii)] $\op{zer}(\mc A+\mc B) \neq \emptyset$ and $\op{zer}(\bar{\mc A}+\bar{\mc B}) \neq \emptyset$.\fineass
\end{enumerate}
\end{remark}


Note that the update of the primal variable $x_i$ in Algorithm~\ref{algo_i} makes use of an approximation $\hat F$ of the pseudogradient mapping $\FF$, since the distribution of the random variable is unknown and hence the expected value mapping can be hard to compute. 
Thus, at each iteration $k$ we let
\begin{equation}\label{eq_F_SA}
\begin{aligned}
\hat F(\bs x^k&,\bs \xi^k)=\op{col}((\hat F_i(\bs x^k,\bs \xi^k))_{i\in\mc I})\\
&=\op{col}(\nabla_{x_1}J_1(\bs x^k,\xi^k_1),\dots,\nabla_{x_N}J_N(\bs x^k,\xi^k_N)),
\end{aligned}
\end{equation}
where $\bs\xi^{k} =\op{col}(\xi_1^{k},\dots,\xi_{N}^k)\in\RR^N$ is a collection of i.i.d. random variables drawn from $\PP$. Essentially, we replace the expected-valued pseudogradient mapping with the realization of \emph{one sample} of the random variable. It follows that 
$$\hat F(\bs x^k,\bs \xi^k)=\FF(\bs x^k)-\delta^k(\bs x^k,\bs \xi^k),$$
where $\delta^k$ is typically called stochastic or approximation error. 
By exploiting the approximation $\hat F$ in \eqref{eq_F_SA} of the expected value mapping $\FF$, we replace the operator $\bar{\mc A}$ in \eqref{eq_expanded} with 
\begin{equation}\label{eq_A_hat}
\hat{\mc A}:\left[\begin{array}{c}\hspace{-.1cm}
(\bs{x},\xi)\hspace{-.1cm} \\ 
\bs z\\
\bs \lambda
\end{array}\right] \hspace{-.1cm}\mapsto\hspace{-.1cm}\left[\begin{array}{c}
\hat F(\bs{x},\xi) \\ 
0\\
\hspace{-.1cm}\bf L\bs\lambda\hspace{-.1cm}
\end{array}\right]+\left[\begin{array}{ccc}
\nabla G(\bs x)^\top\bs\lambda\\
{\bf{L}} \bs z\\
-G(\bs x)-{\bf{L}} \bs z
\end{array}\right].
\end{equation}

Algorithm \ref{algo_i} can now be rewritten in compact form as
\begin{equation}\label{algo}
\bs u^{k+1} =(\op{Id}+\Phi_k^{-1}\bar{\mc B})^{-1}(\bs u^{k}-\Phi_k^{-1}(\hat{\mc A}\bs u^k+\varepsilon^k\bs u^k)),
\end{equation}
where $\Phi_k\succ0$ contains the inverse of step size sequences 
\begin{equation}\label{eq_phi}
\Phi_k=\alpha_k^{-1}\Phi=\alpha_k^{-1}\op{diag}(\gamma^{-1},\nu^{-1},\tau^{-1}),
\end{equation}
with $\gamma^{-1}$, $\nu^{-1}$, $\tau^{-1}$ being diagonal matrices, and $\varepsilon^k=\op{diag}(\epsilon_j^k)\in\RR^{T\times T}$, $T=nN+2Nm$, contains the regularization steps. This last term is what typically characterize the Tikhonov regularization scheme \cite{facchinei2007,koshal2013}.

\subsection{Convergence analysis}\label{sec_conv_vr}

We now study the convergence properties of Algorithm~\ref{algo_i}.
First, to ensure that $\bar{\mc A}$ and $\bar{\mc B}$ have the properties that we use for the analysis, we make the following assumption, which allows us to state the result immediately below.
\begin{assumption}[Monotonicity]\label{ass_mono}
$\FF$ in \eqref{eq_grad} is monotone and $\ell_\FF$-Lipschitz continuous for some $\ell_\FF>0$.
\fineass
\end{assumption}

\begin{lemma}\cite[Lemma~2 and 4]{franci2021}\label{lemma_op}
Let Assumptions \ref{ass_graph} and \ref{ass_mono} hold true, and let $\Phi\succ 0$. Then, the operators $\bar{\mc A}$ and $\bar{\mc B}$ in \eqref{eq_expanded} have the following properties.
\begin{enumerate}
\item $\bar{\mc A}$ is monotone and $\ell_{\bar{\mc A}}$-Lipschitz continuous.
\item $\bar{\mc B}$ is maximally monotone.
\item $\Phi^{-1}\bar{\mc A}$ is monotone and $\ell_{\Phi}$-Lipschitz continuous.
\item $\Phi^{-1}\bar{\mc B}$ is maximally monotone.\fineass
\end{enumerate}
\end{lemma}
The Lipschitz constants in Lemma \ref{lemma_op} depends on those of $g_i$, $i\in\mc I$, and $\FF$ (Assumptions \ref{ass_constr} and \ref{ass_mono}, respectively). Their specific expressions, however, are not relevant for our analysis, and therefore we point to \cite{franci2021} for additional details.
\begin{remark}
Introducing the regularization term makes $\bar{\mc A}+\varepsilon^k$ strongly monotone \cite[Th.~12.2.3]{facchinei2007}. Thus, mere monotonicity of the pseudogradient mapping is enough to show convergence.
In addition, note that the regularization term is added to the operator $\bar{\mc A}$ and not to force strong monotonicity of the SVI in \eqref{eq_SVI}. In fact, starting from a strongly monotone mapping, the resulting operator $\bar{\mc A}$ can be at most cocoercive \cite[Lemma~5 and 7]{yi2019}, \cite[Lemma~2 and 4]{franci2019ecc}.
\fineass
\end{remark}

Taking few samples as in \eqref{eq_F_SA} is realistic and computationally tractable, at the price of requiring an additional assumption on the step sizes. Specifically, we indicate next how to choose the parameters in Algorithm~\ref{algo_i} to ensure that they are vanishing to control the approximation error \cite{koshal2013}. 

\begin{assumption}\label{ass_step}
The step size sequence $(\alpha^k)_{k\in\NN}$ and the regularizing sequences $(\epsilon_j^k)_{k\in\NN}$, $j=1,\dots, T$, are such that $\alpha^k=(k+\eta)^{-a}$ and $\epsilon_j^{k}=(k+\zeta_{j})^{-b}$ for $k \geq 0$, where each $\eta$ and $\zeta_{j}$ are selected from a uniform distribution on the intervals $[\underline{\eta}, \bar{\eta}]$ and $[\zeta, \bar{\zeta}]$, respectively, for some $0<\eta<\bar{\eta}$ and $0<\zeta<\bar{\zeta}$ and $a, b \in(0,1)$, $a+b<1$, and $a>b$. 
\fineass
\end{assumption}

Let us define the filtration $\mc F=\{\mc F_k\}_{k\in\NN}$, that is, a family of $\sigma$-algebras such that $\mathcal{F}_{0} = \sigma\left(X_{0}\right)$ and $\mathcal{F}_{k} = \sigma\left(X_{0}, \xi_{1}, \xi_{2}, \ldots, \xi_{k}\right)$ for all $k \geq 1,$
such that $\mc F_k\subseteq\mc F_{k+1}$ for all $k\in\NN$. In words, $\mc F_k$ contains the information up to iteration $k$.
Since we consider the approximation in \eqref{eq_F_SA}, we let $\Delta^k=\op{col}(\delta^k,0,0)$ be the stochastic error, i.e., $\Delta^k=\bar{\mc A}(\bs u^k)-\hat{\mc A}(\bs u^k,\bs\xi^k)$. Then, an additional assumption is finally needed to regulate its asymptotic behaviour.

\begin{assumption}\label{ass_error}
The step size sequence $(\alpha^k)_{k\in\NN}$, regularization sequence $(\varepsilon^k)_{k\in\NN}$ and the stochastic error $\Delta^k$ satisfy
$\lim _{k \rightarrow \infty}(\alpha^k / \epsilon_{k, \textrm{min} }) \mathbb{E}[\|\Delta^{k}\|_{\Phi}^{2} \mid \mathcal{F}_{k}]=0$ and
$\sum_{k=0}^{\infty} \alpha_{k}^{2} \mathbb{E}[\|\Delta^{k}\|_{\Phi}^{2} \mid \mathcal{F}_{k}]<\infty$ a.s..
\fineass
\end{assumption}

We are now ready to state our main convergence result.

\begin{theorem}\label{theo_sgne}
Let Assumptions \ref{ass_constr} - \ref{ass_error} hold true. Then, the sequence $(\bs x^k)_{k \in \NN}$ generated by Algorithm \ref{algo_i} with $\hat F$ as in \eqref{eq_F_SA} converges a.s. to a v-SGNE of the game in \eqref{eq_game}.  
\fineass
\end{theorem}
\begin{proof}
Convergence to the primal-dual solution follows similarly to \cite{koshal2013} by using the $\Phi$-induced metric. Specifically, with similar steps as \cite[Prop.~1 and 2]{koshal2013} and by letting $\bs y^k$ be the sequence generated by the centralized Tikhonov method \cite[Lemma 3]{koshal2013}, we obtain
$$\begin{aligned}
\EEk{\normsq{\bs u^{k+1}-\bs y^k}_{\Phi}}&\leq (1-c(\alpha^k,\varepsilon^k))\normsq{\bs u^{k+1}-\bs y^k}_{\Phi}\\
&+d(\alpha^k,\varepsilon^k)+\alpha_k^2\EEk{\normsq{\Delta^k}_{\Phi}}
\end{aligned}$$
where $d(\alpha^k,\varepsilon^k)$ and $\alpha_k^2\EEk{\normsq{\Delta^k}}$ vanish as $k\to\infty$ (Assumptions \ref{ass_step} and \ref{ass_error}) and $c(\alpha^k,\varepsilon^k)\in[0,1]$ is such that \cite[Lemma 4.7]{franci2022} can be applied to conclude that $\lim_{k\to\infty}\normsq{\bs u^{k+1}-\bs y^k}=0$.
Then, the statements recalled in Remark \ref{remark_zeros} guarantee convergence of $(\bs x^k)_{k \in \NN}$ to a v-SGNE of the SGNEP in \eqref{eq_game}.
\end{proof}
\begin{remark}
Although our proof follows similar steps to \cite[Prop.~1 and 2]{koshal2013}, note that their component-wise approach can be used basing on the fact that in SNEPs the feasible set reduces to the Cartesian product of the agents' local sets $\Omega_i$, $i\in\mc I$. However, in our generalized setting, the nature of the coupling constraints in \eqref{collective_set} makes this assumption not necessarily true. Moreover, this means that our proof applies to the general case of finding a zero of a monotone inclusion $0\in\bar{\mc A}(\bs u)+\bar{\mc B}(\bs u)$, independently from the fact that such an inclusion comes from the KKT conditions of the game in \eqref{eq_T}. Finally, this also explains why we resort to the step size matrix $\Phi_k$ with structure as defined in \eqref{eq_phi}. 
\fineass
\end{remark}

\subsection{Discussion on the variable step size sequence}\label{sec_discussion}

We traditionally identify two main approaches to perform stochastic approximations: either we take one sample as considered in \eqref{eq_F_SA}, or we take the average over an increasing (possibly \textit{infinite}) number of realizations. The idea behind having a large number of samples is that the variance of the stochastic error disappears with the number of iterations \cite{iusem2017,bot2021}. By following classic results in convergence analysis \cite{yi2019,franci2021,franci2020fb}, however, it turns out that this latter approach is computationally expensive, or even unrealistic in some cases, and hence taking just \emph{one} or a \textit{finite} number of samples is preferable in practice. 
This practical simplification comes at the price of choosing a vanishing step size to control the approximation error, which in our case corresponds to the time-varying matrix $\Phi_k$, thus possibly involving a time-varying metric for the convergence analysis.
Although convergence can be guaranteed in such cases \cite{franci2022,combettes2014}, some additional assumptions on the metric should be satisfied. Specifically, the matrix $\Phi_k$ should be chosen so that
\begin{equation}\label{eq_metric}
\sup _{k \in \mathbb{N}}\|\Phi_k^{-1}\|<\infty \text { and } \forall k >0 \,\,\,(1+\eta_{k}) \Phi_{k+1}^{-1} \succcurlyeq \Phi_k^{-1},
\end{equation}
where $(\eta^k)_{k\in\NN}$ is a nonnegative sequence such that $\sum_k\eta^k<\infty$.
Unfortunately, this contradicts the fact that the step size sequence should be decreasing. Loosely speaking, the motivation for \eqref{eq_metric} stems from the fact that, with a variable metric, it is hard to prove whether the algorithm converges to a zero of the mapping or to a zero of the step sequences. On the other hand, given the specific structure of our matrix $\Phi_k$, we overcome this issue by considering a fixed matrix $\Phi$ that we can use as a metric, and then by pre-multiplying $\Phi$ with a vanishing step as in \eqref{eq_phi}. Note that this formulation allows us to preserve the distributed nature of the algorithm.

We also note that a ``separable" matrix as $\Phi_k$ in \eqref{eq_phi} cannot be used for every iterative distributed algorithm. In the standard SFB \cite{franci2020fb}, for instance, the matrix $\Phi_k$ serves as a preconditioning matrix that has non-zero off-diagonal entries. Pre-multiplying such entries for a quantity (even fixed) would compromise the corresponding KKT conditions, thus making impossible to produce distributed iterations.


\section{Numerical simulations}\label{sec:num_sim}
\begin{figure}[t!]
	\centering
	\includegraphics[width=\columnwidth]{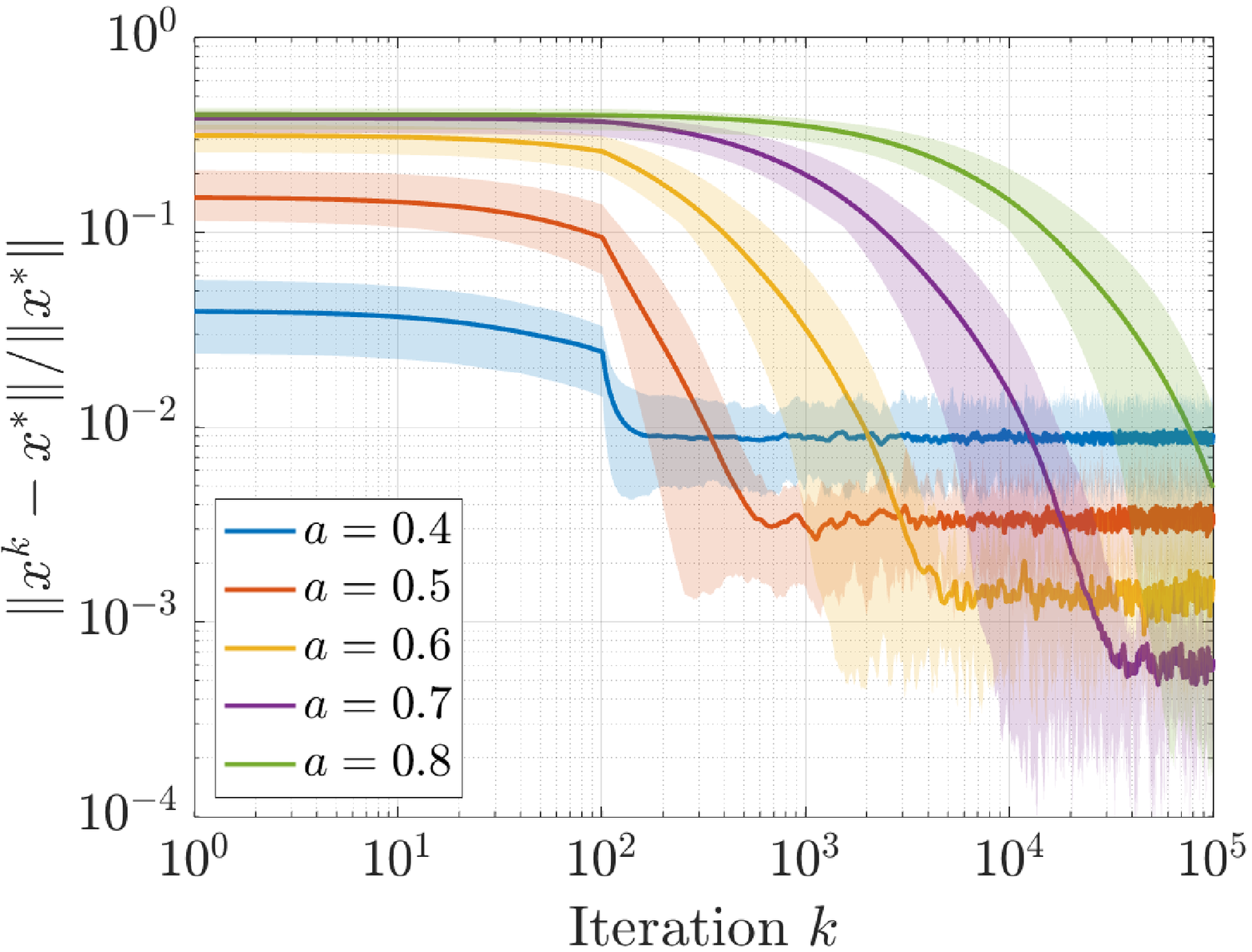}
	\caption{Effect of the parameter $\eta$ of step size, according to Assumption \ref{ass_step}, on the convergence rate.}
	\label{fig_a}
\end{figure}
\begin{figure}[t!]
	\centering
	\includegraphics[width=\columnwidth]{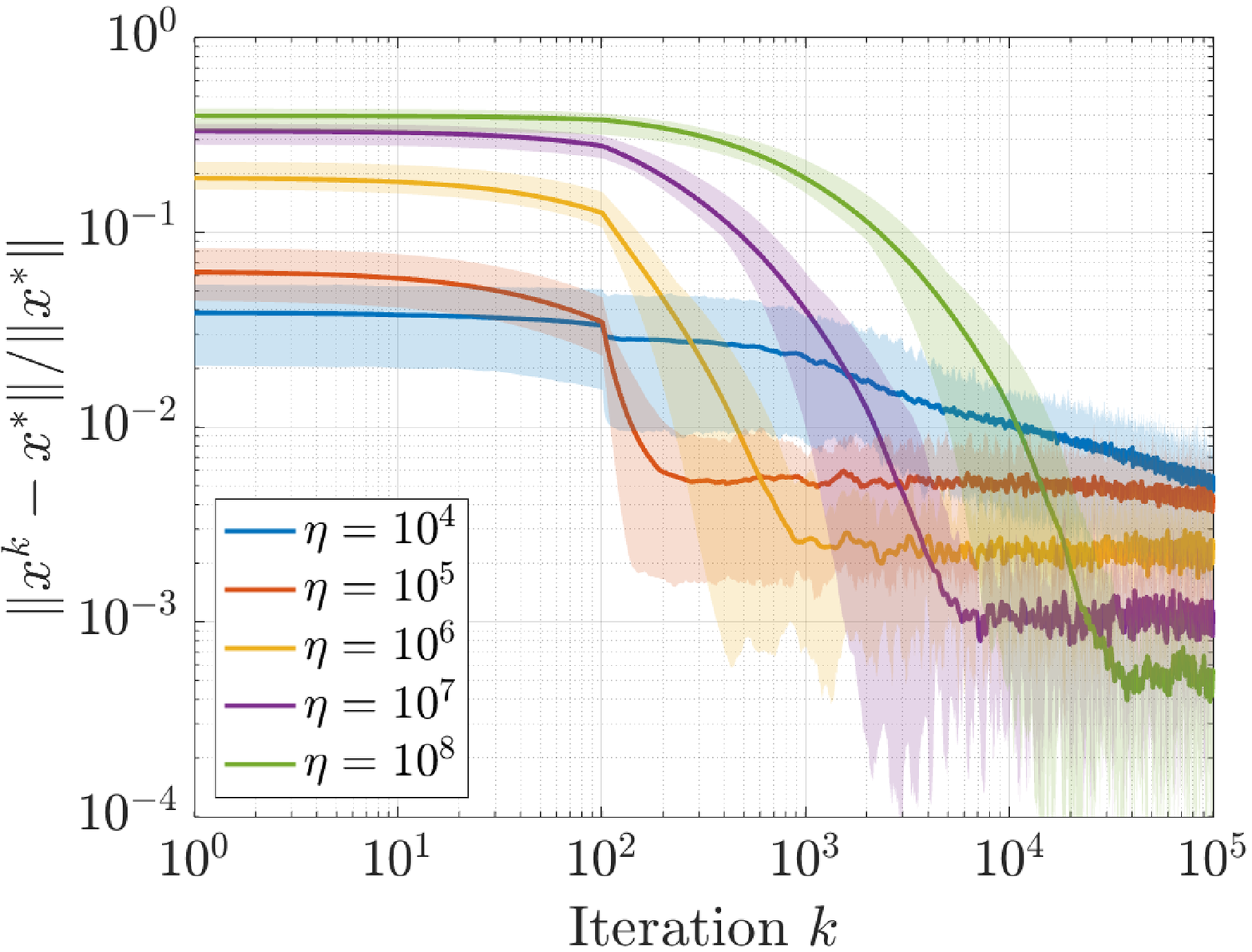}
	\caption{Effect of the parameter $\eta$ of step size, according to Assumption \ref{ass_step}, on the convergence rate.}
	\label{fig_eta}
\end{figure}

\begin{table}
	\caption{Simulation parameters}
	\label{tab:sim_val}
	\centering
	\begin{tabular}{llll}
		\toprule
		Parameter  & Unit & Description   & Value \\
		\midrule
		$\bar{p}$ & \$ & Maximum service price & $35$\\
		$\bar{p}_h$ & \$ & Area price cap & $\sim\mc{U}(0.65  \bar{p}, 0.95 \bar{p})$\\
		$\beta$ & & Commission rate & $0.9$\\
		$\underline{w}$ & \$ & Wage lower bound & $12$\\
		$\theta_i$ &  & Competition parameters & $\sim\mc{U}(0.6, 1)$\\
		$C_h$ &  & Area costumers' demand & $\sim\mc{U}(5, 12)\times10^3$\\
		$K_{i,h}$ &  & Registered drivers & $\sim\mc{U}(0.5, 3)\times10^3$\\
		\bottomrule
	\end{tabular}
\end{table}

We now validate both the model in \S \ref{sec:app} and Algorithm~\ref{algo_i} numerically. Specifically, we consider an instance of the competition among $N = 5$ on-demand ride-hailing firms over $|\mc{H}| = 10$ areas with main parameters in Table~\ref{tab:sim_val}.

\begin{figure}[t!]
	\centering
	\includegraphics[width=\columnwidth]{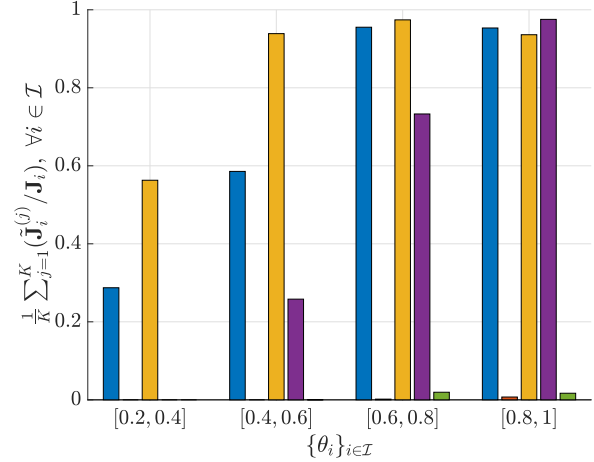}
	\caption{Impact of the substitutability parameter on the averaged percentage of expected profit, for each firm $i \in \mc{I}$.}
	\label{fig:averaged_percentage_profit}
\end{figure}
\begin{figure}[t!]
	\centering
	\includegraphics[width=\columnwidth]{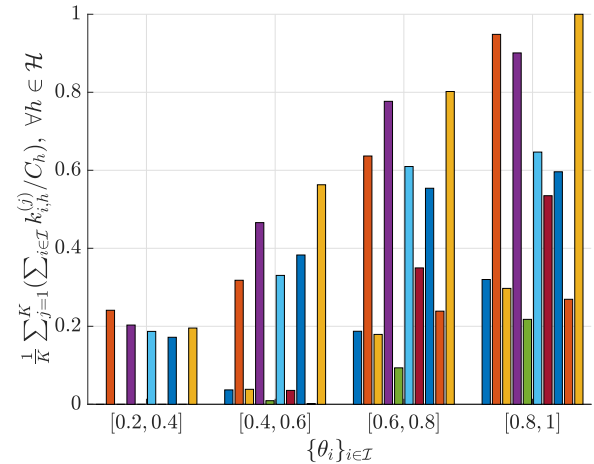}
	\caption{Impact of the substitutability parameter on the averaged costumers' demand satisfaction, for each area $h \in \mc{H}$.}
	\label{fig:averaged_demand_match}
\end{figure}

First, we test the effect of the step size sequence on the convergence of Algorithm~\ref{algo_i} by recalling that, in view of Assumption~\ref{ass_step}, $\alpha^k=(k+\eta)^{-a}$. In particular, Fig.~\ref{fig_a} shows how the rate of convergence is affected by the exponent $a$ while in Fig.~\ref{fig_eta} we plot the effect of the base $\eta$. 
In the first example we choose $\eta = 10^8$ and $\varepsilon^k=(k+10^6)^{-0.15}$, while for the second one $a=0.7$ and $\varepsilon^k=(k+10^6)^{-0.2}$. The thick lines indicate the average performance while the transparent areas are the variability over $10$ runs of the algorithm.


Successively, we examine how the competition parameters $\theta_i$, $i \in \mc{I}$, affect the equilibrium solution of the SGNEP in \eqref{eq:plat_game} against $K = 1000$ random realizations of $\delta_{i,h}$, i.e., the willingness of the driver to provide service. In view of \eqref{eq:driv_service}, given some drivers' expectation level $\delta^{(j)}_{i,h} \sim \mc{U}(\underline{w}, \bar{w}_{i,j})$ and an equilibrium strategy profile $\bs{x}^\ast = \op{col}((x^\ast_{i})_{i \in \mc{I}})$ with $x^\ast_{i} = \op{col}((p^\ast_{i,h})_{h \in \mc{H}})$, if $\delta^{(j)}_{i,h} \leq \beta p^\ast_{i,h}$ then $k^{(j)}_{i,j} = K_{i,j}$. The number of drivers providing a service allows each firm to actually request a payment from the costumers, thus earning $\tilde{\bs{J}}_i^{(j)} \coloneqq \EE_\xi[\sum_{h \in \mc{H}} p^\ast_{i,h} (d^{\textrm{e},(j)}_{i,h}(\xi) - \beta k^{(j)}_{i,h})]$. This value, however, can be different (possibly smaller) from the expected profit $\bs{J}_i \coloneqq \EE_\xi[\sum_{h \in \mc{H}} p^\ast_{i,h} (d^{\textrm{e}}_{i,h}(\xi) - \beta k_{i,h})]$, i.e., the value function associated to the equilibrium condition. Then, for each firm $i \in \mc{I}$, Fig.~\ref{fig:averaged_percentage_profit} shows how the ratio between these two quantities, resulting into a percentage of expected profit, changes with the level of competition in the ride-hailing market. Specifically, we note that when firms operate in an almost oligopoly regime, i.e., $\theta_i \in [0.2, 0.4]$ for all $i \in \mc{I}$, the averaged percentage is small and grows when the market becomes more competitive. For small values of $\theta_i$, indeed, each company is almost independent and tends to select prices to match the lower bound $\underline{w}/\beta$ -- this also coincides with the strategy adopted by firms with only few registered drivers, thus explaining the behavior of firms $2$ and $5$ -- whereas for larger values of $\theta_i$ the firms are entitled to significantly raise their prices, thus allowing to meet drivers' expectations with a higher probability. Such trend is confirmed by the numerical results shown in Fig.~\ref{fig:averaged_demand_match}, which reports the costumers' demand satisfaction, measured as $\sum_{i \in \mc{I}} k^{(j)}_{i,h}/C_h$, for each area $h \in \mc{H}$. In almost oligopoly regimes the costumers' request is met few times only, while it grows significantly when also the competition increases. 

\section{Conclusion}
We have proposed a model for the ride-hailing market under a regulated pricing scenario involving several platforms that compete to offer mobility services. The model takes into account the natural uncertainty of the demand and the need for the platforms to set up suitable pricing strategies to attract riders and fair wages to secure a fleet of drivers. To optimize the operations of these interdependent problems, we have recasted the model as a stochastic Nash equilibrium problem for which we have proposed a distributed, Tikhonov regularization-based algorithm that enjoys convergence guarantees to a Nash equilibrium. In particular, the proposed equilibrium seeking method leverages only a finite number of samples of the uncertainty to perform the stochastic approximation, as well as it requires mere monotonicity of the pseudogradient mapping to establish convergence.

\balance
\bibliographystyle{IEEEtran}
\bibliography{Biblio}

\begin{thebibliography}{10}
\providecommand{\url}[1]{#1}
\csname url@samestyle\endcsname
\providecommand{\newblock}{\relax}
\providecommand{\bibinfo}[2]{#2}
\providecommand{\BIBentrySTDinterwordspacing}{\spaceskip=0pt\relax}
\providecommand{\BIBentryALTinterwordstretchfactor}{4}
\providecommand{\BIBentryALTinterwordspacing}{\spaceskip=\fontdimen2\font plus
\BIBentryALTinterwordstretchfactor\fontdimen3\font minus
  \fontdimen4\font\relax}
\providecommand{\BIBforeignlanguage}[2]{{%
\expandafter\ifx\csname l@#1\endcsname\relax
\typeout{** WARNING: IEEEtran.bst: No hyphenation pattern has been}%
\typeout{** loaded for the language `#1'. Using the pattern for}%
\typeout{** the default language instead.}%
\else
\language=\csname l@#1\endcsname
\fi
#2}}
\providecommand{\BIBdecl}{\relax}
\BIBdecl

\bibitem{nyc_ridehailing}
$\vphantom{}${NYC} OpenData. (2022) For-hire vehicle base aggregate report.
  [Online] \texttt{https://opendata.cityofnewyork.us/}.

\bibitem{yi2019}
P.~Yi and L.~Pavel, ``An operator splitting approach for distributed
  generalized {Nash} equilibria computation,'' \emph{Automatica}, vol. 102, pp.
  111--121, 2019.

\bibitem{franci2020fb}
B.~Franci and S.~Grammatico, ``A distributed forward-backward algorithm for
  stochastic generalized {Nash} equilibrium seeking,'' \emph{IEEE Transactions
  on Automatic Control}, 2020.

\bibitem{robbins1951}
H.~Robbins and S.~Monro, ``A stochastic approximation method,'' \emph{The
  Annals of Mathematical Statistics}, pp. 400--407, 1951.

\bibitem{koshal2013}
J.~Koshal, A.~Nedic, and U.~V. Shanbhag, ``Regularized iterative stochastic
  approximation methods for stochastic variational inequality problems,''
  \emph{IEEE Transactions on Automatic Control}, vol.~58, no.~3, pp. 594--609,
  2013.

\bibitem{iusem2017}
A.~Iusem, A.~Jofr{\'e}, R.~I. Oliveira, and P.~Thompson, ``Extragradient method
  with variance reduction for stochastic variational inequalities,'' \emph{SIAM
  Journal on Optimization}, vol.~27, no.~2, pp. 686--724, 2017.

\bibitem{bot2021}
R.~I. Bo{\c{t}}, P.~Mertikopoulos, M.~Staudigl, and P.~T. Vuong, ``Minibatch
  forward-backward-forward methods for solving stochastic variational
  inequalities,'' \emph{Stochastic Systems}, vol.~11, no.~2, pp. 112--139,
  2021.

\bibitem{franci2021}
B.~Franci and S.~Grammatico, ``Stochastic generalized {Nash} equilibrium
  seeking in merely monotone games,'' \emph{IEEE Transactions on Automatic
  Control}, 2021.

\bibitem{franci2021sprg}
------, ``Distributed projected--reflected--gradient algorithms for stochastic
  generalized {Nash} equilibrium problems,'' in \emph{2021 European Control
  Conference (ECC)}.\hskip 1em plus 0.5em minus 0.4em\relax IEEE, 2021, pp.
  369--374.

\bibitem{facchinei2007}
F.~Facchinei and J.-S. Pang, \emph{Finite-dimensional variational inequalities
  and complementarity problems}.\hskip 1em plus 0.5em minus 0.4em\relax
  Springer Science \& Business Media, 2007.

\bibitem{kannan2012}
A.~Kannan and U.~V. Shanbhag, ``Distributed computation of equilibria in
  monotone {Nash} games via iterative regularization techniques,'' \emph{SIAM
  Journal on Optimization}, vol.~22, no.~4, pp. 1177--1205, 2012.

\bibitem{facchinei2010}
F.~Facchinei and C.~Kanzow, ``Generalized {{Nash}} equilibrium problems,''
  \emph{Annals of Operations Research}, vol. 175, no.~1, pp. 177--211, 2010.

\bibitem{zhong2022demand}
Y.~Zhong, T.~Yang, B.~Cao, and T.~Cheng, ``On-demand ride-hailing platforms in
  competition with the taxi industry: {P}ricing strategies and government
  supervision,'' \emph{International Journal of Production Economics}, vol.
  243, p. 108301, 2022.

\bibitem{he2020off}
E.~J. He, S.~Savin, J.~Goh, and C.-P. Teo, ``Off-platform threats in on-demand
  services,'' \emph{Available at SSRN 3550646}, 2020.

\bibitem{bimpikis2019spatial}
K.~Bimpikis, O.~Candogan, and D.~Saban, ``Spatial pricing in ride-sharing
  networks,'' \emph{Operations Research}, vol.~67, no.~3, pp. 744--769, 2019.

\bibitem{beer2017qualitative}
R.~Beer, C.~Brakewood, S.~Rahman, and J.~Viscardi, ``Qualitative analysis of
  ride-hailing regulations in major {American} cities,'' \emph{Transportation
  Research Record}, vol. 2650, no.~1, pp. 84--91, 2017.

\bibitem{mcguire1983industry}
T.~W. McGuire and R.~Staelin, ``An industry equilibrium analysis of downstream
  vertical integration,'' \emph{Marketing Science}, vol.~2, no.~2, pp.
  161--191, 1983.

\bibitem{PANDEY2019269}
V.~Pandey, J.~Monteil, C.~Gambella, and A.~Simonetto, ``On the needs for {MaaS}
  platforms to handle competition in ridesharing mobility,''
  \emph{Transportation Research Part C: Emerging Technologies}, vol. 108, pp.
  269--288, 2019.

\bibitem{fabiani2021personalized}
F.~Fabiani, A.~Simonetto, and P.~J. Goulart, ``Personalized incentives as
  feedback design in generalized {Nash} equilibrium problems,'' \emph{IEEE
  Transactions on Automatic Control}, 2021, (Under review -- available at
  {\footnotesize \texttt{arxiv.org/abs/2203.12948}}).

\bibitem{belgioioso2017}
G.~Belgioioso and S.~Grammatico, ``Semi-decentralized {Nash} equilibrium
  seeking in aggregative games with separable coupling constraints and
  non-differentiable cost functions,'' \emph{IEEE Control Systems Letters},
  vol.~1, no.~2, pp. 400--405, 2017.

\bibitem{ravat2011}
U.~Ravat and U.~V. Shanbhag, ``On the characterization of solution sets of
  smooth and nonsmooth convex stochastic {Nash} games,'' \emph{SIAM Journal on
  Optimization}, vol.~21, no.~3, pp. 1168--1199, 2011.

\bibitem{shapiro2021}
A.~Shapiro, D.~Dentcheva, and A.~Ruszczynski, \emph{Lectures on stochastic
  programming: modeling and theory}.\hskip 1em plus 0.5em minus 0.4em\relax
  SIAM, 2021.

\bibitem{palomar2010}
D.~P. Palomar and Y.~C. Eldar, \emph{Convex optimization in signal processing
  and communications}.\hskip 1em plus 0.5em minus 0.4em\relax Cambridge
  university press, 2010.

\bibitem{facchinei2007vi}
F.~Facchinei, A.~Fischer, and V.~Piccialli, ``On generalized {Nash} games and
  variational inequalities,'' \emph{Operations Research Letters}, vol.~35,
  no.~2, pp. 159--164, 2007.

\bibitem{auslender2000}
A.~Auslender and M.~Teboulle, ``Lagrangian duality and related multiplier
  methods for variational inequality problems,'' \emph{SIAM Journal on
  Optimization}, vol.~10, no.~4, pp. 1097--1115, 2000.

\bibitem{franci2019ecc}
B.~Franci and S.~Grammatico, ``A damped forward--backward algorithm for
  stochastic generalized {Nash} equilibrium seeking,'' in \emph{2020 European
  Control Conference (ECC)}.\hskip 1em plus 0.5em minus 0.4em\relax IEEE, 2020,
  pp. 1117--1122.

\bibitem{franci2022}
\BIBentryALTinterwordspacing
------, ``Convergence of sequences: A survey,'' \emph{Annual Reviews in
  Control}, 2022. [Online]. Available:
  \url{https://www.sciencedirect.com/science/article/pii/S1367578822000037}
\BIBentrySTDinterwordspacing

\bibitem{combettes2014}
P.~L. Combettes and B.~C. V{\~u}, ``Variable metric forward--backward splitting
  with applications to monotone inclusions in duality,'' \emph{Optimization},
  vol.~63, no.~9, pp. 1289--1318, 2014.

\end{thebibliography}

\end{document}